\documentclass{birkjour}
\usepackage{upref,amsmath,amssymb,amsthm,mathrsfs,amsfonts,mathtools}
\usepackage{hyperref}
\newtheorem{theorem}{Theorem}[section]
\newtheorem{corollary}[theorem]{Corollary}
\newtheorem{lemma}[theorem]{Lemma}
\newtheorem{proposition}[theorem]{Proposition}
\theoremstyle{definition}
\newtheorem{definition}[theorem]{Definition}
\theoremstyle{remark}
\newtheorem{remark}[theorem]{Remark}
\newtheorem{example}[theorem]{Example}
\numberwithin{equation}{section}

\newcommand{\Hilb}{\mathcal{H}}
\newcommand{\alfa}{\mathbb{A}}
\newcommand{\env}{C^*_e}
\newcommand{\masa}{\mathbb{D}}

\newcommand{\vhta}{\mathbb{B}}

\newcommand{\cpct}{\mathcal{K}}
\newcommand{\mmax}{\mathbb{M}}

\hypersetup{ colorlinks=true, linkcolor=black, filecolor=black, urlcolor=black, citecolor=black }

\begin{document}

\title[Unital Operator Spaces and Discrete Groups]
 {Unital Operator Spaces and Discrete Groups}

\author{Nikolaos Koutsonikos-Kouloumpis}

\address{
University of Patras\\
Department of Mathematics\\
26504 Patras\\
Greece}

\email{up1019669@ac.upatras.gr}

\thanks{This work was financially supported by the "Andreas Mentzelopoulos Foundation".}

\subjclass{Primary: 46L07; Secondary: 47A15, 47L25, 47L35}

\keywords{Masa bimodules, unital operator spaces, operator systems, reflexive spaces, CSL algebras, $C^*$-envelopes}

\date{September 2024}

\begin{abstract}
We introduce the trivial intersection property for concrete operator spaces and we show that a unital space with this property has no nontrivial boundary ideals. We provide various examples of such spaces, among which are strongly reflexive masa bimodules and completely distributive CSL algebras. We show that unital operator spaces acting on $\ell^2(\Gamma)$ for any set $\Gamma$, that contain the masa $\ell^\infty(\Gamma)$, possess the trivial intersection property, and we use this to prove that a unital surjective complete isometry between such spaces is a unitary equivalence. Then, we apply these results to $w^*$-closed $\ell^\infty(G)$-bimodules acting on $\ell^2(G)$ for a group $G$ and we relate them to algebraic properties of $G$.
\end{abstract}

\maketitle

\section{Introduction and Preliminaries}
\label{intro}

In this note, our spaces of interest are concrete unital operator spaces, i.e. norm-closed subspaces $\mathcal{U}\subseteq B(\Hilb)$ that contain the unit of $B(\Hilb)$. Here, $B(\Hilb)$ denotes the space of all bounded linear operators acting on a Hilbert space $\Hilb$.

In Section \ref{envelope}, we focus our attention on concrete unital operator spaces and we introduce a property called the \textbf{trivial intersection property}. It follows directly that every unital subspace $\mathcal{U}\subseteq B(\Hilb)$ with this property has no boundary ideals, hence its $C^*$-envelope $\env(\mathcal{U})$ coincides with the $C^*$-algebra $\mathcal{U}$ generates in $B(\Hilb)$, which is denoted by $C^*(\mathcal{U})$. It turns out that this property appears in several different cases. To name a few, every strongly reflexive masa bimodule has the trivial intersection property, as well as every completely distributive CSL algebra. We also show that every subspace of $B(\ell^2(\Gamma))$ that contains the masa $\ell^\infty(\Gamma)$ has the trivial intersection property and we use this to prove Theorem \ref{unitar}, which states that every unital surjective complete isometry between unital spaces that contain a discrete masa is a unitary equivalence.

Finally, in Section \ref{disc}, we apply results from Section \ref{envelope} to subspaces of $B(\ell^2(G))$, where $G$ is a group. We are interested in spaces of the form $\mmax(\Omega)$ (see relation (\ref{mmax}) in Section \ref{disc}), i.e. the space of all operators supported in a set $\Omega\subseteq G\times G$. The support of an operator was first introduced by Arveson in \cite{arveson} and generalized by Erdos, Katavolos and Shulman in \cite{katavolos}, and our definition for $\mmax(\Omega)$ is based on their work.  In fact, we mostly deal with sets of the form $E^\star$ for $E\subseteq G$, where $E^\star=\{(g,h)\in G\times G:g^{-1}h\in E\}$ and we investigate how algebraic properties of $E$ are related to algebraic and topological properties of $\mmax(E^\star)$. 

Arveson introduced in \cite{arveson} the notion of a \textbf{Commutative Subspace Lattice} (CSL). By definition, a CSL is a complete sublattice of projections acting on a Hilbert space that commute with each other and which contains the zero and the identity operator. A \textbf{CSL algebra} is any algebra of the form Alg$(\mathcal{L}):=\{T\in B(\Hilb):  P^\perp TP=0 \  \text{for each}  \ P\in\mathcal{L}\}$, where $\mathcal{L}$ is a CSL. It is well-known that $\mathscr{A}\subseteq B(\Hilb)$ is a CSL algebra if and only if it is a reflexive algebra (in the sense of Halmos \cite{halmos}) that contains a \textbf{Maximal Abelian Self-adjoint Algebra (masa)}, i.e. an abelian self-adjoint subalgebra of $B(\Hilb)$, which is not properly contained in another abelian self-adjoint subalgebra of $B(\Hilb)$. As a generalization of CSL algebras, masa bimodules have been widely studied by Arveson, Davidson, Eleftherakis, Erdos, Katavolos, Ludwig, Shulman, Todorov, Turowska et al. (see e.g. \cite{arveson}, \cite{nest}, \cite{eleftherakis}, \cite{katavolos}, \cite{ludwig}, \cite{shurowska}, \cite{shulman}, \cite{todorov}, \cite{turowska}).

The definition of the $C^*$-envelope, as well as its universal property is due to M. Hamana \cite{hamana}, although W. Arveson had already proved their existence and the universal property in several cases \cite{arv1,arv2}. For an explicit definition of the $C^*$-envelope and description of the universal property, the reader is referred to \cite[Theorem 4.3.1]{blecher}. We shall need the fact that for every concrete unital operator space $\mathcal{U}\subseteq B(\Hilb)$, its $C^*$-envelope $C^*_e(\mathcal{U})$ is ($*$-isomorphic to) $C^*(\mathcal{U})/J$ where $J$ (called the \textbf{Shilov boundary} of $\mathcal{U}$) is the largest closed two-sided ideal of $C^*(\mathcal{U})$ with the property that the canonical map 
\[
q\circ i:\mathcal{U}\rightarrow C^*(\mathcal{U})/J
\]is a complete isometry, where $i:\mathcal{U}\rightarrow C^*(\mathcal{U})$ is the inclusion map and $q$ is the quotient map from $C^*(\mathcal{U})$ onto $C^*(\mathcal{U})/J$ (any closed two sided ideal of $C^*(\mathcal{U})$ with this property is called a \textbf{boundary ideal} for $\mathcal{U}$). For a proof of the above fact, see e.g. \cite[4.3.2]{blecher}. For a further discussion on (unital) operator spaces, see \cite{blecher}, \cite{ruan}, \cite{paulsen} or \cite{pisier}.

For $A\subseteq B(\Hilb)$, we denote by $[A]$ the linear span of the set $A$, i.e. all finite linear combinations of elements in $A$. The \textbf{commutant} of $A$ is the set $A':=\{T\in B(\Hilb):TS=ST \   \text{for all} \   S\in A\}$. We also denote by $\overline{A}^{w^*}$, $\overline{A}^{WOT}$ etc. the closure of $A$ in the corresponding topologies ($w^*$-topology, Weak Operator Topology etc.).

For a set $\Gamma$, we view the elements of $\ell^2(\Gamma)$ as functions $f:\Gamma\rightarrow \mathbb{C}$. Therefore,
\[
\sum_{\gamma\in \Gamma}|f(\gamma)|^2<\infty
\]for all $f\in\ell^2(\Gamma)$. The inner product in $\ell^2(\Gamma)$ is
\[
\langle f,h\rangle =\sum_{\gamma\in \Gamma}f(\gamma)\overline{h(\gamma)}.
\]Every function $f\in\ell^\infty(\Gamma)$ induces an associated operator $M_f\in B(\ell^2(\Gamma))$ by $M_f(g)=fg$ for all $g\in\ell^2(\Gamma)$. So, by $\ell^\infty(\Gamma)$ we mean all operators of this kind. It is not so difficult to see that this is actually a masa acting on $\ell^2(\Gamma)$. For all subsets $A\subseteq \Gamma$, we write $\chi_A$ for the characteristic function on $A$, and $M_A:=M_{\chi_A}$ which is a projection in $\ell^\infty(\Gamma)$. 

In all Sections, all subspaces are assumed to be norm-closed.

\section{$C^*$-envelopes and the trivial intersection property}
\label{envelope}

In this section, we introduce a property, called the \textbf{trivial intersection property} for concrete operator spaces and we show that the $C^*$-envelope of unital subspaces $\mathcal{U}\subseteq B(\Hilb)$ with this property is ($*$-isomorphic to) $C^*(\mathcal{U})$. In Theorem \ref{enve}, we indicate some conditions, under which a space has the trivial intersection property and as we will see, Theorem \ref{enve} can be applied in order to compute the $C^*$-envelope in several different cases of unital operator spaces. 

\begin{definition}
    We say that a closed subspace $\mathcal{U}\subseteq B(\Hilb)$ has the \textbf{trivial intersection property} if for every closed two-sided ideal $J$ of $C^*(\mathcal{U})$, $\mathcal{U}\cap J=\{0\}$ implies $J=\{0\}$.
\end{definition}

We note that the trivial intersection property is similar to the already known ideal intersection property, defined for $C^*$-subalgebras rather than subspaces of $C^*$-algebras. For further details, see e.g. \cite{brown}, \cite{exel} or \cite{groupoid}.

In order to fix some notation, for two $C^*$-algebras $C,D$, by $C\cong D$ we mean that $C$ and $D$ are $*$-isomorphic. We denote by $\mathcal{K}(\Hilb)$ the set of all compact operators in $B(\Hilb)$ and for every subset $V\subseteq B(\Hilb)$ we set $\mathcal{K}(V):=\mathcal{K}(\Hilb)\cap V$.

\begin{proposition}
\label{tiproperty}
    If $\mathcal{U}\subseteq B(\Hilb)$ is a closed unital subspace that has the trivial intersection property, then $\env(\mathcal{U})\cong C^*(\mathcal{U})$.
\end{proposition}

\begin{proof}
    Let $J$ be a boundary ideal of $\mathcal{U}$, i.e. the restriction $\pi:=q|_\mathcal{U}$ of the quotient map $q$ from $C^*(\mathcal{U})$ onto $C^*(\mathcal{U})/J$ is a complete isometry. In particular, $\pi$ is one-to-one, therefore $\mathcal{U}\cap J=\{0\}$, so by the trivial intersection property, $J=\{0\}$. This proves that the only boundary ideal of $\mathcal{U}$ is $\{0\}$, therefore $\env(\mathcal{U})\cong C^*(\mathcal{U})$.
\end{proof}

\begin{theorem}
\label{enve}
    Suppose that $\mathcal{U}\subseteq B(\Hilb)$ is a unital subspace and that there exists a subset $\mathcal{S}$ of $\mathcal{K}(C^*(\mathcal{U}))$ such that $\mathcal{S}C^*(\mathcal{U})'\subseteq\mathcal{U}$ and $I\in\overline{[\mathcal{S}+\mathcal{S}^*]}^{WOT}$. Then, $\mathcal{U}$ has the trivial intersection property.
\end{theorem}

\begin{proof}
    Suppose $J$ is an ideal of $C^*(\mathcal{U})$ with $J\cap\mathcal{U}=\{0\}$ and $q$ is the quotient map from $C^*(\mathcal{U})$ onto $C^*(\mathcal{U})/J$. Then $\overline{J}^{w^*}$ is an ideal of $\alfa=\overline{C^*(\mathcal{U})}^{w^*}$ and the latter is a von Neumann algebra. By \cite[Proposition II.3.12]{takesaki}, $\overline{J}^{w^*}=\alfa p$ for some projection $p$ in the center of $\alfa$. Clearly $p\in\overline{J}^{w^*}$, so by the Kaplansky density Theorem, there exists a net $(a_i)\subseteq J$ with $\|a_i\|\leq 1$ for each $i$, such that $a_i\to p$ strongly. It is then easy to see that $Ka_i\to Kp$ in the norm topology, for each $K\in \mathcal{S}$. Since the center of $\alfa$ is contained in $C^*(\mathcal{U})'$, we have $Kp\in \mathcal{U}$. Therefore, using that $Ka_i\in J$, we see that $0=q(Ka_i)\to q(Kp)$, thus $Kp=0$ for every $K\in \mathcal{S}$. Since $p\in C^*(\mathcal{U})'$, we also have $\mathcal{S}^*p=0$, thus $p=0$ by our hypothesis that $I\in\overline{[\mathcal{S}+\mathcal{S}^*]}^{WOT}$ and $J=\{0\}$.
\end{proof}

Theorem \ref{enve} provides a large class of Corollaries/Examples of concrete unital operator spaces, with the trivial intersection property. 

\begin{corollary}
    Suppose $\mathcal{U}\subseteq B(\Hilb)$ is a unital subspace with the property that $\overline{\cpct(\mathcal{U})+\cpct(\mathcal{U})^*}^{WOT}=B(\Hilb)$. Then $\mathcal{U}$ has the trivial intersection property.  
\end{corollary}

\begin{proof}
    Since $\cpct(\mathcal{U})+\cpct(\mathcal{U})^*\subseteq C^*(\mathcal{U})$, we have $C^*(\mathcal{U})'=\mathbb{C}I$. Thus, for $\mathcal{S}=\cpct(\mathcal{U})$, all conditions of Theorem \ref{enve} are met. 
\end{proof}

A subspace $\mathcal{U}\subseteq B(\Hilb)$ is called \textbf{strongly reflexive} if there exists a set $\mathscr{R}$ of rank one operators such that $\mathcal{U}=$ Ref$(\mathscr{R})$ (the \textbf{reflexive hull} of $\mathscr{R}$ is defined as Ref$(\mathscr{R}):=\{T\in B(\Hilb):Tx\in\overline{[\mathscr{R}x]}\text{ for all $x\in \Hilb$}\}$). We have the following:

\begin{corollary}
\label{strongly}
    Suppose $\mathcal{U}\subseteq B(\Hilb)$ is strongly reflexive and a unital $\masa$-bimodule over a masa $\masa\subseteq B(\Hilb)$. Then $\mathcal{U}$ has the trivial intersection property.
\end{corollary}

\begin{proof}
    It follows from \cite[Theorem 2.4]{katavolos}, applied on the subspace spanned by the rank one operators in $\mathcal{U}$, that the latter is WOT-dense in $\mathcal{U}$. Therefore, $\mathcal{U}=\overline{\cpct(\mathcal{U})}^{WOT}$. Using that $\mathcal{U}$ is unital, we have $\masa\subseteq\mathcal{U}\subseteq C^*(\mathcal{U})$, hence $C^*(\mathcal{U})'\subseteq \masa'=\masa$. Thus, for $\mathcal{S}=\cpct(\mathcal{U})$, we have $\mathcal{S}C^*(\mathcal{U})'\subseteq \mathcal{U}\masa\subseteq \mathcal{U}$ and $I\in \overline{\mathcal{S}}^{WOT}$ hence, by Theorem \ref{enve}, $\mathcal{U}$ has the trivial intersection property.
\end{proof}

\begin{corollary}
\label{cor1}
    Let $\mathcal{V}\subseteq B(\Hilb)$ be any set satisfying $\mathcal{V}'\subseteq \mathcal{V}$ and let $\mathcal{U}$ be a unital subspace with
    $\mathcal{U}\subseteq \overline{\cpct(\mathcal{U})+\cpct(\mathcal{U})^*}^{WOT}$ and $\mathcal{V}\mathcal{U}\subseteq \mathcal{U}$ $($or $\mathcal{U}\mathcal{V}\subseteq\mathcal{U})$. Then $\mathcal{U}$ has the trivial intersection property.
\end{corollary}

\begin{proof}
    Clearly, $\mathcal{V}\subseteq\mathcal{U}$, therefore $C^*(\mathcal{U})'\subseteq \mathcal{V}$. Then, if we set $\mathcal{S}=\cpct(\mathcal{U})$, we get $\mathcal{S}C^*(\mathcal{U})'\subseteq \mathcal{U}$ (in either one of the cases $\mathcal{V}\mathcal{U}\subseteq\mathcal{U}$, $\mathcal{U}\mathcal{V}\subseteq\mathcal{U}$). Finally, $I\in\mathcal{U}\subseteq\overline{\cpct(\mathcal{U})+\cpct(\mathcal{U})^*}^{WOT}$, therefore we can apply Theorem \ref{enve}.
\end{proof}

\begin{example}
    If $\mathcal{U}\subseteq B(\Hilb)$ is a unital $\masa$-bimodule over some masa $\masa\subseteq B(\Hilb)$, which satisfies $\mathcal{U}\subseteq\overline{\cpct(\mathcal{U})+\cpct(\mathcal{U})^*}^{WOT}$ then by Corollary \ref{cor1}, $\mathcal{U}$ has the trivial intersection property.
\end{example}

\begin{example}
    Let $\mathcal{L}$ be a CSL acting on a Hilbert space $\Hilb$ and set $\mathcal{U}=\overline{\text{Alg}(\mathcal{L})+\cpct(\Hilb)}^{\|\cdot\|}$. Since $\mathcal{L}'\subseteq $ Alg$(\mathcal{L})$, we have $\mathcal{L}'\mathcal{U}\subseteq \mathcal{U}$. Since $\mathcal{L}\subseteq\mathcal{L}'$ and $\cpct(\mathcal{U})=\cpct(\Hilb)$, we can apply Corollary \ref{cor1} with $\mathcal{V}=\mathcal{L}'$.
\end{example}

For the last case, we refer the reader to \cite[Chapter 23]{nest} for the definition of a \textbf{completely distributive lattice}.

\begin{corollary}
    Every completely distributive CSL algebra has the trivial intersection property. 
\end{corollary}

\begin{proof}
    Let $\mathcal{U}=\textup{Alg}(\mathcal{L})$, where $\mathcal{L}$ is a completely distributive CSL. By the discussion at the end of Section 8 in \cite{erdos}, it follows that $\mathcal{U}$ is strongly reflexive. It is well-known that every CSL algebra contains a masa, hence $\mathcal{U}$ is also a masa bimodule, so we can apply Corollary \ref{strongly}.
\end{proof}

For the rest of the section, we deal with subspaces of $B(\ell^2(\Gamma))$, where $\Gamma$ is any set, which contain the discrete masa $\ell^\infty(\Gamma)$ and we shall prove the following:

\begin{theorem}
\label{unitar}
    Suppose $\mathcal{U}_1,\mathcal{U}_2$ are unital operator subspaces of $B(\ell^2(\Gamma_1))$ and $B(\ell^2(\Gamma_2))$ respectively, and each $\mathcal{U}_i$ contains the masa $\masa_i:=\ell^\infty(\Gamma_i)$. Let $\theta: \mathcal{U}_1\rightarrow \mathcal{U}_2$ be a unital surjective complete isometry. Then, $\theta$ is a unitary equivalence, i.e. there exists a unitary $W:\ell^2(\Gamma_1)\rightarrow\ell^2(\Gamma_2)$ such that $\theta(T)=WTW^*$ for all $T\in \mathcal{U}_1$. 
\end{theorem}

For the proof, we need the following Lemma, which is a direct consequence of Theorem \ref{enve} and Proposition \ref{tiproperty}:

\begin{lemma}
\label{main2}
    Suppose $\mathcal{U}$ is a unital operator space acting on $\Hilb=\ell^2(\Gamma)$ that contains the masa $\masa=\ell^\infty(\Gamma)$. Then, $\env(\mathcal{U})\cong C^*(\mathcal{U})$.
\end{lemma}

\begin{proof}
    Let $\mathcal{S}=\cpct(\masa)\subseteq\cpct((C^*(\mathcal{U}))$. Since $C^*(\mathcal{U})'\subseteq \masa$, we have $C^*(\mathcal{U})'\mathcal{S}\subseteq\masa\subseteq \mathcal{U}$. Furthermore, $I\in\overline{\mathcal{S}}^{w^*}$, thus by Theorem \ref{enve}, $\mathcal{U}$ has the trivial intersection property and then $\env(\mathcal{U})\cong C^*(\mathcal{U})$ by Proposition \ref{tiproperty}.
\end{proof}

\begin{proof}[Proof of Theorem \ref{unitar}]
    We set $\alfa:=C^*(\mathcal{U}_1)$ and $\vhta:=C^*(\mathcal{U}_2)$. By Lemma \ref{main2}, these are $*$-isomorphic to $C^*_e(\mathcal{U}_1)$ and $C^*_e(\mathcal{U}_2)$ respectively, hence, by the universal property of the $C^*$-envelope applied to both $\theta:\mathcal{U}_1\to\vhta$ and $\theta^{-1}:\mathcal{U}_2\to\alfa$, there exists a $*$-algebra isomorphism $\tilde{\theta}:\alfa\rightarrow \vhta$ that extends $\theta:\mathcal{U}_1\to \vhta$. Since $\masa_1\subseteq \mathcal{U}_1\subseteq \alfa$ and $\masa_2\subseteq \mathcal{U}_2\subseteq\vhta$, we have $\alfa'\subseteq \masa_1'=\masa_1\subseteq \alfa$ and similarly $\vhta'\subseteq\vhta$, hence $\alfa'$ and $\vhta'$ are in fact the centers of $\alfa$ and $\vhta$ respectively. An algebra isomorphism restricts to an isomorphism between the centers, hence $\tilde{\theta}$ maps $\alfa'$ onto $\vhta'$. Since $\alfa'\subseteq \ell^\infty(\Gamma_1)$ and $\vhta'\subseteq\ell^\infty(\Gamma_2)$, it is well-known that there exists a unique family $\{Q_i:i\in \mathcal{I}\}$ of minimal orthogonal projections in $\alfa'$ such that $\sum_iQ_i=I$. Since $\tilde{\theta}$ is an order isomorphism, it follows that $\{\tilde{\theta}(Q_i) :i\in \mathcal{I}\}$ is a set of minimal orthogonal projections in $\vhta'$ with $\sum_i\tilde{\theta}(Q_i)=I$. It is then not so hard to see that 
    \begin{gather*}
    \overline{\alfa}^{w^*}=\bigoplus_iQ_iB(\ell^2(\Gamma_1))Q_i \\
    \overline{\vhta}^{w^*}=\bigoplus_i\tilde{\theta}(Q_i)B(\ell^2(\Gamma_2))\tilde{\theta}(Q_i).
    \end{gather*}
    For each $i$, we set $\alfa_i:=Q_i\alfa Q_i\subseteq\alfa$, $\vhta_i:=\tilde{\theta}(Q_i)\vhta\tilde{\theta}(Q_i)\subseteq \vhta$. Since $Q_i\in \ell^\infty(\Gamma_1)$ and $\tilde{\theta}(Q_i)\in\ell^\infty(\Gamma_2)$, for each $i$ there exist unique subsets $A_i\subseteq\Gamma_1$ and $B_i\subseteq\Gamma_2$ such that $Q_i=M_{A_i}$ and $\tilde{\theta}(Q_i)=M_{B_i}$, so we can view $\alfa_i$ and $\vhta_i$ as $C^*$-subalgebras of $B(\ell^2(A_i))$ and $B(\ell^2(B_i))$ respectively. We observe that
    \[
    \overline{\alfa_i}^{w^*}=Q_i\overline{\alfa}^{w^*}Q_i=B(\ell^2(A_i))
    \]and by the same reasoning $\overline{\vhta_i}^{w^*}=B(\ell^2(B_i))$, so both $\alfa_i$ and $\vhta_i$ are irreducible. Moreover, since $\alfa_i$ and $\vhta_i$ contain $\ell^\infty(A_i)$ and $\ell^\infty(B_i)$ respectively, which in turn contain non-zero compact operators, by \cite[Corollary I.10.4]{davidson}, $\mathcal{K}(\ell^2(A_i))\subseteq \alfa_i$ and $\mathcal{K}(\ell^2(B_i))\subseteq\vhta_i$. Thus, by \cite[Theorem 10.4.6]{kadison}, there exists a unitary operator $W_i:\ell^2(A_i)\rightarrow \ell^2(B_i)$ such that   
    \[
    \tilde{\theta}(T)=W_iTW_i^*
    \]for every $T\in \alfa_i$. Finally, if we define $W:=\oplus_iW_i:\ell^2(\Gamma_1)\rightarrow\ell^2(\Gamma_2)$, which is unitary, we get $\tilde{\theta}(T)=WTW^*$ for every $T\in \alfa$. 
\end{proof}

\begin{remark}
\label{vniso}
    Suppose $\alfa\subseteq B(\ell^2(\Gamma_1))$, $\vhta\subseteq B(\ell^2(\Gamma_2))$ are $*$-isomorphic von Neumann algebras that contain the masas $\ell^\infty(\Gamma_1)$ and $\ell^\infty(\Gamma_2)$ respectively. Then, by Theorem \ref{unitar}, they are unitarily equivalent. In this case though, the unitary operator $U$ that induces the equivalence can be picked so that $\ell^\infty(\Gamma_2)=U\ell^\infty(\Gamma_1)U^*$. In order to see this, let $\theta:\alfa\rightarrow \vhta$ be a $*$-isomorphism. Then, following the proof of Theorem \ref{unitar}, we obtain an orthogonal family $\{Q_i:i\in \mathcal{I}\}$ of minimal projections in $\alfa'$ such that $\sum_iQ_i=I$ and $\{\theta(Q_i):i\in \mathcal{I}\}$ is also an orthogonal family of minimal projections in $\vhta'$ with $\sum_i\theta(Q_i)=I$, and since $\alfa$, $\vhta$ are von Neumann algebras,
    \begin{gather*}
        \alfa=\bigoplus_iQ_iB(\ell^2(\Gamma_1))Q_i \\
    \vhta=\bigoplus_i\theta(Q_i)B(\ell^2(\Gamma_2))\theta(Q_i).
    \end{gather*}
    For certain subsets $A_i\subseteq\Gamma_1$ and $B_i\subseteq\Gamma_2$, we can view $Q_iB(\ell^2(\Gamma_1))Q_i$ and $\theta(Q_i)B(\ell^2(\Gamma_2))\theta(Q_i)$ as $B(\ell^2(A_i))$ and $B(\ell^2(B_i))$ respectively, and again by the proof of Theorem \ref{unitar}, we get that $\ell^2(A_i)$ and $\ell^2(B_i)$ are isomorphic, therefore $A_i$ and $B_i$ have the same cardinality. Let $f_i:A_i\rightarrow B_i$ be one-to-one and onto. If $\{e_\gamma:\gamma\in\Gamma_1\}$ and $\{e_\delta:\delta\in\Gamma_2\}$ are the canonical bases of $\ell^2(\Gamma_1)$ and $\ell^2(\Gamma_2)$ respectively, for each $i$ we set $U_i:\ell^2(A_i)\rightarrow\ell^2(B_i)$ by $U_ie_\gamma=e_{f_i(\gamma)}$, for each $\gamma\in A_i$. It is clear then that $\vhta=U\alfa U^*$ and $\ell^\infty(\Gamma_2)=U\ell^\infty(\Gamma_1)U^*$, where $U=\oplus_i U_i$, and that $Ue_\gamma=e_{f(\gamma)}$, where $f(\gamma)=f_i(\gamma)$ for the unique $i$ such that $\gamma\in A_i$.
\end{remark}

\section{Discrete Groups}
\label{disc}

In this section, we shall apply results of the previous section to discrete groups. For a group $G$ and $f,h\in\ell^2(G)$, we write $fh^*$ for the rank one operator
\[
x\mapsto \langle x,h\rangle f.
\]For every $g\in G$, we set $e_g:=\chi_{\{g\}}\in \ell^2(G)$. For any subset $\Omega\subseteq G\times G$, we set
\begin{equation}
\label{mmax}
    \mmax(\Omega):=\{T\in B(\ell^2(G)):  M_F TM_E=0 \ \text{for all} \ E,F\subseteq G \ \text{with} \ E\times F \subseteq\Omega^c\}.
\end{equation}
Clearly, $\mmax(\Omega)$ is a $w^*$-closed subspace of $B(\ell^2(G))$, which is also a $\ell^\infty(G)$-bimodule. 

The following Lemma is well-known, although, to the best of our knowledge, an explicit proof is missing from the literature:   
\begin{lemma}
\label{prod}
    For any $\Omega\subseteq G\times G$, we have
    \[
    \mmax(\Omega)=\overline{[e_he_g^* : \ (g,h)\in\Omega]}^{w^*}.
    \]
\end{lemma}

\begin{proof}
Using an identical argument as in \cite[Section 6.3, paragraph 8]{discretemasa}, it can be verified that
\begin{equation}
\label{mmaxeq}
    \mmax(\Omega)=\{T\in B(\ell^2(G)):\langle Te_g,e_h\rangle=0 \ \text{for all} \ (g,h)\in\Omega^c\}.
\end{equation}
Since $\mmax(\Omega)$ is a $w^*$-closed subspace, it follows that
\[
\overline{[e_he_g^* : \ (g,h)\in\Omega]}^{w^*}\subseteq \mmax(\Omega).
\]The reverse inclusion follows again from  (\ref{mmaxeq}) and the fact that 
\[
T=\sum_{g,h}(e_he_h^*)T(e_ge_g^*)=\sum_{g,h}\langle Te_g,e_h\rangle e_he_g^*
\]for all $T\in B(\ell^2(G))$.
\end{proof}

Now, let $E\subseteq G$. We set $E^\star:=\{(g,h)\in G\times G:g^{-1}h\in E\}$. The following proposition relates $E$ with $\mmax(E^\star)$ regarding their algebraic properties:

\begin{proposition}
\label{gr}
    Suppose $G$ is a group with identity element $1$ and $E\subseteq G$. Then:
    \begin{itemize}
        \item[(i)] $1\in E$ if and only if $\mmax(E^\star)$ is unital.
        \item[(ii)] $E=E^{-1}$ if and only if $\mmax(E^\star)$ is self-adjoint.
        \item[(iii)] $E\cdot E\subseteq E$ if and only if $\mmax(E^\star)$ is an algebra.
    \end{itemize}
    In particular, $E$ is a subgroup of $G$ if and only if $\mmax(E^\star)$ is a von Neumann algebra.
\end{proposition}

\begin{proof}
    \begin{itemize}
        \item[(i)]We have $1\in E$ if and only if $\{(g,g):g\in G\}\subseteq E^\star$. On the other hand, since $M_{F_1} M_{F_2} =0\Leftrightarrow F_1\cap F_2=\emptyset$, we have $I\in\mmax(E^\star)$ if and only if $F_1\cap F_2=\emptyset$ for all $(F_1\times F_2)\cap E^\star=\emptyset$, which is equivalent to the fact $\{(g,g):g\in G\}\subseteq E^\star$.
        \item[(ii)]$E=E^{-1} \Leftrightarrow$ $E^\star$ is symmetric, i.e. $(g,h)\in E^\star\Leftrightarrow (h,g)\in E^\star$. Thus $(F_1\times F_2)\cap E^\star=\emptyset $ if and only if $(F_2\times F_1)\cap E^\star=\emptyset$, which is equivalent to $\mmax(E^\star)$ being self-adjoint.
        \item[(iii)]Suppose $E\cdot E\subseteq E$. By Lemma \ref{prod}, it suffices to show that if $(g_1,h_1)$, $(g_2,h_2)\in E^\star$, the operator $(e_{h_1}e_{g_1}^*)(e_{h_2}e_{g_2}^*)\in \mmax(E^\star)$. Since
        \[
        (e_{h_1}e_{g_1}^*)(e_{h_2}e_{g_2}^*)=\langle e_{h_2},e_{g_1}\rangle e_{h_1}e_{g_2}^*
        \]it suffices to assume $g_1=h_2$. Since $g_1^{-1}h_1\in E$ and $g_2^{-1}h_2\in E$, we have
        \[
        g_2^{-1}h_2g_1^{-1}h_1=g_2^{-1}h_1\in E
        \]therefore $(g_2,h_1)\in E^\star$ and $e_{h_1}e_{g_2}^*\in \mmax(E^\star)$. Now suppose that $\mmax(E^\star)$ is an algebra. For $g,h\in E$, we have $(g^{-1},1)$, $(1,h)\in E^\star$, therefore $e_1e_{g^{-1}}^*$ and $e_he_1^*$ are in $\mmax(E^\star)$. Since the latter is an algebra, 
        \[
        e_he_{g^{-1}}^*=(e_he_1^*)(e_1e_{g^{-1}}^*)\in\mmax(E^\star)
        \]hence $(g^{-1},h)\in E^\star$, or $gh\in E$. \qedhere
    \end{itemize} \end{proof}
For every set $E\subseteq G$, we denote by $\langle E\rangle$ the subgroup generated by $E$, and for a subset $A\subseteq B(\ell^2(G))$, we denote by $W^*(A)$ the von Neumann algebra generated by $A$. The following is a corollary of Proposition \ref{gr}:
\begin{corollary}
\label{gener}
    For every nonempty set $E\subseteq G$, we have
    \[
    W^*(\mmax(E^\star))=\mmax(\langle E\rangle^\star).
    \]
\end{corollary}

\begin{proof}
    By Proposition \ref{gr}, $\mmax(\langle E\rangle^\star)$ is a von Neumann algebra, which clearly contains $\mmax(E^\star)$, therefore $W^*(\mmax(E^\star))\subseteq\mmax(\langle E\rangle^\star)$. 
    
    For the reverse inclusion, by Lemma \ref{prod}, it suffices to prove that for every $(g,h)\in \langle E\rangle^\star$, $e_he_g^*\in W^*(\mmax(E^\star))$. We have $g^{-1}h\in\langle E\rangle$, so there exist $x_1,..,x_n\in E\cup E^{-1}$ such that $g^{-1}h=x_1..x_n$, i.e. $h=gx_1..x_n$. Hence, it suffices to show that $e_{gx_1..x_k}e_g^*\in W^*(\mmax(E^\star))$ for all $k\geq 1$ and $x_1,..,x_k\in E\cup E^{-1}$. We show this by induction:

    For $k=1$, let $x_1\in E\cup E^{-1}$. In case $x_1\in E$, $(g,gx_1)\in E^\star$, thus $e_{gx_1}e_g^*\in \mmax(E^\star)$. In case $x_1\in E^{-1}$, $(gx_1,g)\in E^\star$, which implies $e_ge_{gx_1}^*\in\mmax(E^\star)$, thus $e_{gx_1}e_g^*\in W^*(\mmax(E^\star))$.

    Now suppose the claim is true for some $ k\geq 1$. Let $x_1,..,x_{k+1}\in E\cup E^{-1}$. Then, $e_{gx_1..x_k}e_g^*\in W^*(\mmax(E^\star))$. In case $x_{k+1}\in E$, $(gx_1..x_k,gx_1..x_kx_{k+1})\in E^\star$, therefore $e_{gx_1..x_kx_{k+1}}e_{gx_1..x_k}^*\in \mmax(E^\star)$ and thus
    \[
    e_{gx_1..x_kx_{k+1}}e_g^*=(e_{gx_1..x_kx_{k+1}}e_{gx_1..x_k}^*)(e_{gx_1..x_k}e_g^*)\in W^*(\mmax(E^\star)).
    \]If $x_{k+1}\in E^{-1}$, $(gx_1..x_kx_{k+1},gx_1..x_k)\in E^\star$, so $e_{gx_1..x_k}e_{gx_1..x_kx_{k+1}}^*\in\mmax(E^\star)$ and therefore
    \[
    e_{gx_1..x_kx_{k+1}}e_g^*=(e_{gx_1..x_k}e_{gx_1..x_kx_{k+1}}^*)^*(e_{gx_1..x_k}e_g^*)\in W^*(\mmax(E^\star)).
    \]The proof is now complete.
\end{proof}

\begin{theorem}
\label{cosets}
   Suppose $H_1$, $H_2$ are subgroups of $G_1$ and $G_2$ respectively. The following are equivalent:
    \begin{itemize}
        \item[(i)] $\mmax(H_1^\star)$ and $\mmax(H_2^\star)$ are unitarily equivalent.
        \item[(ii)] $\mmax(H_1^\star)$ and $\mmax(H_2^\star)$ are $*$-isomorphic.
        \item[(iii)] $H_1$ and $H_2$ have the same cardinality and the same indices in $G_1$, $G_2$ respectively, i.e. $|H_1|=|H_2|$ and $[G_1:H_1]=[G_2:H_2]$.
    \end{itemize}
\end{theorem}

\begin{proof}
     Obviously, $(i)\Rightarrow (ii)$. Now suppose that $(ii)$ holds. By Remark \ref{vniso}, there exists a bijection $f:G_1\rightarrow G_2$ and a unitary $U:\ell^2(G_1)\rightarrow\ell^2(G_2)$ such that $\mmax(H_2^\star)=U\mmax(H_1^\star)U^*$ and $Ue_g=e_{f(g)}$ for each $g\in G_1$. We have the following:
    \begin{gather*}
        g^{-1}h\in H_1\iff (g,h)\in H_1^\star \iff e_h e_g^*\in \mmax(H_1^\star) \\
        \iff U(e_h e_g^*)U^*\in\mmax(H_2^\star)\iff e_{f(h)} e_{f(g)}^*\in\mmax(H_2^\star) \\
        \iff (f(g),f(h))\in H_2^\star \iff f(g)^{-1}f(h)\in H_2.
    \end{gather*}
Using this, we see that the map
\[
\tilde{f}:G_1/H_1\rightarrow G_2/H_2
\]
\[
gH_1\mapsto f(g)H_2
\]is well-defined, one-to-one and onto. This means $[G_1:H_1]=[G_2:H_2]$. Also, by our earlier observation, $g\in H_1$ if and only if $f(1)^{-1}f(g)\in H_2$, that is $f(g)\in f(1)H_2$. Since $f$ is one-to-one and onto, we get $f(H_1)=f(1)H_2$. Therefore,
\[
|H_1|=|f(H_1)|=|f(1)H_2|=|H_2|
\]and we have proved $(ii)\Rightarrow(iii)$.

Finally, suppose that $|H_1|=|H_2|$ and $[G_1:H_1]=[G_2:H_2]$. Then, there exists a bijective map $f:G_1\rightarrow G_2$ such that $f(gH_1)=f(g)H_2$ for every $g\in G_1$. This implies $g^{-1}h\in H_1 \iff f(g)^{-1}f(h)\in H_2$ for all $g,h\in G_1$. Now, if we define $U:\ell^2(G_1)\rightarrow\ell^2(G_2)$ by $Ue_g=e_{f(g)}$, using Lemma \ref{prod} one easily observes that $\mmax(H_2^\star)=U\mmax(H_1^\star)U^*$, therefore $(iii)\Rightarrow (i)$.
\end{proof}
A direct Corollary using Lagrange's Theorem is the following:  
\begin{corollary}
\label{finite}
    Suppose $G$ is a finite group and $H_1$, $H_2$ are subgroups of $G$. The following are equivalent:
    \begin{itemize}
        \item[(i)] $\mmax(H_1^\star)$ and $\mmax(H_2^\star)$ are unitarily equivalent.
        \item[(ii)] $\mmax(H_1^\star)$ and $\mmax(H_2^\star)$ are $*$-isomorphic.
        \item[(iii)] $|H_1|=|H_2|$.
    \end{itemize}
\end{corollary}

\begin{remark}
    In case $G$ is finite and $H_1$, $H_2$ are two subgroups isomorphic to each other, by Corollary \ref{finite}, we clearly have $\mmax(H_1^\star)\cong\mmax(H_2^\star)$. However, this is not true in general for infinite groups. For example, if $G=\mathbb{Z}$, $H_1=\mathbb{Z}$, $H_2=2\mathbb{Z}$, the groups $H_1$, $H_2$ are isomorphic, but $[G:H_1]=1$, $[G:H_2]=2$, so by Theorem \ref{cosets}, $\mmax(H_1^\star)\not\cong\mmax(H_2^\star)$. Also, $\mmax(H_1^\star)\cong\mmax(H_2^\star)$ does not mean in general that $H_1$ and $H_2$ are isomorphic (even in the finite case). For example, for $G=\mathbb{Z}_2\times\mathbb{Z}_4$, the subgroups $H_1=\langle(0,1)\rangle$, $H_2=\langle(0,2),(1,0)\rangle$ are not isomorphic, but they both have 4 elements, therefore by Corollary \ref{finite}, $\mmax(E_1^\star)\cong\mmax(E_2^\star)$.
\end{remark}

\begin{corollary}
    Let $E_1$, $E_2$ be subsets of the groups $G_1$, $G_2$ respectively, both containing the identity elements, and suppose $\mmax(E_1^\star)$ and $\mmax(E_2^\star)$ are completely isometrically isomorphic via a unital map. Then, $|\langle E_1\rangle|=|\langle E_2\rangle |$ and $[G_1:\langle E_1\rangle]=[G_2:\langle E_2\rangle]$.
\end{corollary}

\begin{proof}
    By Proposition \ref{gr}, each $\mmax(E_i^\star)$ contains the identity, and since it is a $\ell^\infty(G_i)$-bimodule, it contains $\ell^\infty(G_i)$. Therefore, by (the proof of) Theorem \ref{unitar}, the $C^*$-algebras $C^*(\mmax(E_1^\star))$, $C^*(\mmax(E_2^\star))$ are unitarily equivalent, hence so are their $w^*$-closures, i.e. $W^*(\mmax(E_1^\star))$ and $W^*(\mmax(E_2^\star))$. By Corollary \ref{gener}, we get $\mmax(\langle E_1\rangle^\star)\cong\mmax(\langle E_2\rangle^*)$ and the conclusion follows from Theorem \ref{cosets}.
\end{proof}


\subsection*{Acknowledgment}
The author would like to thank Professor G. K. Eleftherakis for his great guidance and his suggestions for this note, as well as Professor A. Katavolos for his very useful observations.
\bigskip

\subsection*{Declarations}
\smallskip
\textbf{Data availability statement} Data sharing is not applicable to this article as no data sets were generated or analyzed during this study.
\bigskip

\noindent
\textbf{Conflict of interest} The author declares that there are no conflicts of interest.


\begin{thebibliography}{1}
\bibitem{arv1} W. Arveson, \textit{Subalgebras of C*-algebras.} Acta Math. \textbf{123} (1969), 141--224.
\bibitem{arv2} W. Arveson, \textit{Subalgebras of C*-algebras II.} Acta Math. \textbf{128} (1972), 271--308.
\bibitem{arveson} W. Arveson, \textit{Operator algebras and invariant subspaces.} Ann. Math. (2) \textbf{100} (1974), 433--532. 
\bibitem{blecher} D. P. Blecher and C. L. Merdy, \textit{Operator algebras and their modules: an operator space approach.} London Mathematical Society Monographs. New Series. Clarendon Press, Oxford, 2004.
\bibitem{brown} J. H. Brown, A. H. Fuller, D. R. Pitts and S. A. Reznikoff, \textit{Regular ideals, ideal intersections, and quotients.}  Integral Equations Oper. Theory \textbf{96} (2024), 31 p.
\bibitem{nest} K. R. Davidson, \textit{Nest algebras. Triangular forms for operator algebras on Hilbert space.} Pitman Research Notes in Mathematics Series. Longman Scientific \& Technical, Harlow, Essex, 1988.
\bibitem{davidson} K. R. Davidson, \textit{C*-algebras by example.} Fields Institute Monographs. AMS, American Mathematical Society, Providence, RI, 1996.
\bibitem{ruan} E. G. Effros and Z. J. Ruan, \textit{Operator spaces.} London Mathematical Society Monographs. New Series. Clarendon Press, Oxford, 2000.
\bibitem{eleftherakis} G. K. Eleftherakis, \textit{Bilattices and Morita equivalence of MASA bimodules.} Proc. Edinb. Math. Soc. \textbf{59} (2016), 605--621.
\bibitem{erdos} J. A. Erdos, \textit{Reflexivity for subspace maps and linear spaces of operators.} Proc. Lond. Math. Soc. (3) \textbf{52} (1986), 582--600.
\bibitem{katavolos} J. A. Erdos, A. Katavolos and V. S. Shulman, \textit{Rank one subspaces of bimodules over maximal abelian selfadjoint algebras.} J. Funct. Anal. \textbf{157} (1998), 554--587.
\bibitem{exel} R. Exel, \textit{Regular ideals under the ideal intersection property.} arXiv:2301.10073 (2023).
\bibitem{halmos} P. R. Halmos, \textit{Reflexive lattices of subspaces.} J. Lond. Math. Soc., II. Ser. \textbf{4} (1971), 257--263.
\bibitem{hamana} M. Hamana, \textit{Injective envelopes of operator systems.} Publ. Res. Inst. Math. Sci. \textbf{15} (1979), 773--785.
\bibitem{kadison} R. V. Kadison and J. R. Ringrose, \textit{Fundamentals of the theory of operator algebras. Volume II: Advanced theory.} Pure and Applied Mathematics (Academic Press). Academic Press, New York, NY, 1997.
\bibitem{discretemasa} A. Katavolos, \textit{Operator algebras: an introduction.} Serdica Math. J. \textbf{41} (2015), 49-82.
\bibitem{groupoid} M. Kennedy, S. J. Kim, X. Li, S. Raum and D. Ursu, \textit{The ideal intersection property for essential groupoid C*-algebras.} arXiv:2107.03980 (2021).
\bibitem{ludwig} J. Ludwig and L. Turowska, \textit{On the connection between sets of operator synthesis and sets of spectral synthesis for locally compact groups.} J. Funct. Anal. \textbf{233} (2006), 206--227.
\bibitem{paulsen} V. Paulsen, \textit{Completely bounded maps and operator algebras.} Cambridge Studies in Advanced Mathematics. Cambridge University Press, Cambridge, 2002.
\bibitem{pisier} G. Pisier, \textit{Introduction to operator space theory.} London Mathematical Society Lecture Note Series. Cambridge University Press, Cambridge. London Mathematical Society, London, 2003.
\bibitem{shurowska} V. Shulman and L. Turowska, \textit{Operator synthesis. I: Synthetic sets, bilattices and tensor algebras.} J. Funct. Anal. \textbf{209} (2004), 293--331.
\bibitem{shulman} V. S. Shulman, I. G. Todorov and L. Turowska, \textit{Reduced spectral synthesis and compact operator synthesis.} Adv. Math. \textbf{367} (2020), 49 p.
\bibitem{takesaki} M. Takesaki, \textit{Theory of operator algebras I.} Springer-Verlag, New York, Heidelberg, Berlin, 1979.
\bibitem{todorov} I. G. Todorov, \textit{Spectral synthesis and masa-bimodules.} J. Lond. Math. Soc., II. Ser. \textbf{65} (2002), 733--744.
\bibitem{turowska} I. G. Todorov and L. Turowska, \textit{Transference and preservation of uniqueness.} Isr. J. Math. \textbf{230} (2019), 1--21.
\end{thebibliography}
\end{document}